\documentclass[11pt]{amsart}
\usepackage{amsmath}
\usepackage{amssymb}
\usepackage{amsthm}
\usepackage{enumitem}
\usepackage{varioref} 

\newtheorem{definition}{Definition}[section]
\newtheorem{example}[definition]{Example}
\newtheorem{proposition}[definition]{Proposition}
\newtheorem{theorem}[definition]{Theorem}
\newtheorem{lemma}[definition]{Lemma}

\newtheorem{remark}[definition]{Remark}

\DeclareMathOperator{\JSymbol}{\mathcal{J}}

\def\R{\mathbb R}

\DeclareMathOperator{\LspaceSymbol}{\mathbf{L}}
\newcommand{\Lspace}[1]{\LspaceSymbol^{#1}}

\newcommand{\Lpspace}[1][p]{\Lspace{{#1}}}
\newcommand{\LGspace}[1][G]{\Lspace{{#1}}}

\DeclareMathOperator{\WspaceSymbol}{\mathbf{W}}
\newcommand{\Wspace}[1]{\WspaceSymbol^{#1}}
\newcommand{\WTspace}[1]{\WspaceSymbol^{#1}_T}

\newcommand{\WLGspace}[1][G]{{\Wspace{1}}\Lspace{{#1}}}
\newcommand{\WTLGspace}[1][G]{{\WTspace{1}}\Lspace{{#1}}}

\newcommand{\norm}[1]{\|#1\|}

\newcommand{\LGnorm}[2][G]{\norm{#2}_{\LGspace[{#1}]}}

\newcommand{\WLGnorm}[2][G]{\norm{#2}_{\WLGspace[{#1}]}}
\newcommand{\WLGnormother}[2][G]{\norm{#2}_{1,\WLGspace[{#1}]}}

\newcommand{\inner}[2]{\langle #1,#2\rangle}

\newcommand{\imb}[2]{C_{#1,#2 }}
\newcommand{\imbinftG}{\imb{\infty}{\WLGspace}}

\newcommand{\modularG}[1]{R_G(#1)}
\usepackage{amsmath}

\title[Periodic solutions in anisotropic Orlicz-Sobolev space ]{
Mountain pass type periodic solutions for Euler-Lagrange equations in anisotropic
Orlicz-Sobolev space
}

\author{M. Chmara}

\author{J. Maksymiuk}

\address{Department of Technical Physics and Applied Mathematics, Gda\'{n}sk University of
Technology, Narutowicza 11/12, 80-952 Gda\'{n}sk, Poland}
\email{magdalena.chmara@pg.edu.pl, jakub.maksymiuk@pg.edu.pl}

\begin{document}

\begin{abstract}
Using the Mountain Pass Theorem, we establish the existence of periodic solution for
Euler-Lagrange equation.   Lagrangian consists of kinetic part (an anisotropic G-function),
potential part $K-W$ and a forcing term.
We consider two situations: $G$ satisfying \ref{delta2}$\cap$\ref{nabla2} in infinity and globally. We give conditions on the growth of the potential near zero for both situations.
\end{abstract}


\maketitle
\section{Introduction}
We consider the  second order system
\begin{equation}
\tag{ELT}
\label{eq:ELT}
\begin{cases}
\frac{d}{dt}L_v(t,u(t),\dot u(t))=L_x(t,u(t),\dot u(t))\quad \text{ for a.e. }t\in[-T,T]\\
u(-T)=u(T)
\end{cases}
\end{equation}
where $L\colon [-T,T]\times\R^N\times\R^N\to \R$ is given by
\[
L(t,x,v)=G(v)+V(t,x)+\inner{f(t)}{x}.
\]
We assume that $G$ is a differentiable G-function and $V$ satisfies suitable growth conditions.
If $G(v)=\frac1p|v|^p$ then the equation \eqref{eq:ELT} reduces to p-laplacian.
More general case is $G(v)=\phi(|v|)$, where $\phi$ is convex and nonnegative. In the above cases, $G$ depends on norm $|v|$ and its growth is the
same in all directions (isotropic).  In this paper we consider the situation when the growth of $G$ is different in different directions (anisotropic) e.g. $G(x,y)=|x|^p+|y|^q$.

Existence of periodic solutions for the problem \eqref{eq:ELT} was investigated in many papers, e.g.: \cite{AciMaz17}
(anisotropic case), \cite{AciBurGiuMazSch15} (isotropic case), \cite{PasWan16} ($(p,q)$-laplacian), \cite{XuTan07,MaZha09}  ($p$-laplacian),  \cite{MawWil89}  
(laplacian) and many others.

This paper is motivated by \cite{Dao16,Ter12,IzyJan05}, where
the existence of homoclinic solution of $\frac{d}{dt}L_v(t,u(t),\dot u(t))=L_x(t,u(t),\dot u(t))$ is investigated (see also \cite{ZelRab91, LvLu12}).
In all these papers an intermediate step is to show, using the Mountain Pass Theorem, that corresponding periodic problem has a solution.

We want to adapt methods from \cite{Dao16} to anisotropic Orlicz-Sobolev space setting.
It turns out, that the Mountain pass geometry of action functional  is strongly depended on Simonenko indices $p_G$ and $q_G$  (see section \ref{sec:Gfacts}).  
To show that the action  functional satisfies the Palais-Smale condition we need index $q_G^{\infty}$. 
Similar observation can be found in \cite{Cle04, Bar17, Cle00, Le02} where the existence of elliptic systems via the  Mountain Pass Theorem is considered. In \cite{Bar17} authors deal with an anisotropic problem. The isotropic case is considered in \cite{Cle04, Cle00, Le02}.

We assume that:
\begin{enumerate}[label=$(A_{\arabic*})$]
\item
\label{V:G}

$G\colon \R^N\to[0,\infty)$ is a continuously differentiable G-function (i.e. G is convex, even,
$G(0)=0$ and $G(x)/|x|\to \infty$, as $|x|\to \infty$) satisfying $\Delta_2$ and $\nabla_2$
condition,

\item
\label{V:C1K-W}
$V(t,x)=K(t,x)-W(t,x)$, where $K,W\in \mathbf{C}^1([-T,T]\times\R^N,\R)$,

\item
\label{V:near0}
there exist $a\in\Lpspace[1]([-T,T],\R)$, $b>1$ and $\rho_0>0$ such that
\[
V(t,x)\geq b\,G(x)-a(t)\quad \text{ for $|x|\leq\rho_0$, $t\in[-T,T]$},
\]

\item\label{V:infty}
there exist $b_1>0$ and $p>1$ satisfying $|\cdot|^p\prec G$, such that
\[
\liminf_{|x|\to\infty}\frac{K(t,x)}{|x|^{p}}\geq b_1
\text{ uniformly in } t\in[-T,T]
\]
and
\[
\liminf_{|x|\to\infty}\frac{W(t,x)}{\max\{K(t,x),G(x)\}}>3
\text{ uniformly in } t\in[-T,T],
\]

\item\label{V:AR}	there exist $\nu\in\R$, $\mu>q_G^{\infty}+\nu$ and $\kappa\in
\Lpspace[1]([-T,T],[0,\infty))$ such that
\[
\inner{V_x(t,x)}{x}\leq (q_G^{\infty}+\nu)K(t,x)-\mu W(t,x)+\kappa(t)\quad
\text{ for $(t,x)\in[-T,T]\times\R^N$,}
\]
\item \label{V:0}$\int_{-T}^{T}V(t,0)\,dt=0$,

\item\label{f:G*normf<Cf} $f\in\LGspace[G^\ast]([-T,T],\R^N)$.
\end{enumerate}

Assumption \ref{V:near0}, \ref{V:infty} and \ref{V:AR} are  essential for the Mountain Pass Theorem. We need \ref{V:near0} to show that there exists $\alpha>0$ such that functional
\begin{equation}
\tag{$\JSymbol$}
\JSymbol(u)=\int_{-T}^{T} G(\dot u)+V(t,u)+\inner{f}{u}\,dt.
\end{equation} is greater than $\alpha$ on the boundary of some ball (see lemma \ref{lem:J>alpha_rho<1}).
To do this we need to control behavior of V near zero.

Condition \ref{V:infty} allows us to control the growth of $V$ at infinity.
The first condition, together with \ref{V:AR}, is used to show that functional satisfies the Palais-Smale condition.
The latter condition is used to show that potential is negative far from zero.
Assumption \ref{V:AR} is a modification of the well known Ambrosetti-Rabinowitz condition.

Let us denote by $\imbinftG$ an embedding constant for $\WLGspace\hookrightarrow
\Lpspace[\infty]$ and define
\begin{equation}
\label{eq:rho} \tag{$\rho$}
\rho:=\frac{\rho_0}{\imb{\infty}{\WTLGspace}}.
\end{equation}

Now we can formulate our main results.

\begin{theorem}
\label{thm:main_rho0<1}
Let $L\colon [-T,T]\times\R^N\times\R^N\to\R$ satisfies \ref{V:G}--\ref{f:G*normf<Cf}. Assume
that $G$ satisfies $\Delta_2$ and $\nabla_2$ globally, and
\begin{equation}
\label{thm1:as1}
R_{G^{\ast}}(f) + \int_{-T}^{T} a(t)\,dt
<\min\{1,b-1\}\begin{cases}	(\rho/2)^{q_G},&\rho\leq 2\\(\rho/2)^{p_G}&\rho> 2.\end{cases}
\end{equation}

Then \eqref{eq:ELT} posses the periodic solution.
\end{theorem}

The assumption that $G$ satisfies \ref{delta2} and \ref{nabla2} globally  can be relaxed if we
assume that $\rho\geq 2$ but in this case wee need stronger assumption on $f$.

\begin{theorem}
\label{thm:main_rho>2}
Let  $L\colon [-T,T]\times\R^N\times\R^N\to\R$ satisfies  \ref{V:G}--\ref{f:G*normf<Cf}. Assume
that $\rho\geq 2$ and
\begin{equation}
\label{frho>2}
\quad
R_{G^{\ast}}(f) + \int_{-T}^{T}a(t)\,dt
<
\min\{1,b-1\}(\rho/2)
\end{equation}
Then \eqref{eq:ELT} posses the periodic solution.
\end{theorem}

Theorem \ref{thm:main_rho0<1} generalizes Lemma 3.1 from \cite{Dao16}. 
Actually, assumption \eqref{thm1:as1} has the same form as $(H_5)$ in \cite{Dao16}, since  in p-laplacian case $p_G=q_G=q_G^{\infty}=p$.
Note that p-laplacian satisfies  \ref{delta2} and \ref{nabla2} globally.
To the best authors knowledge there is no analogue of Theorem \ref{thm:main_rho>2} in the literature.

Now we give two examples of potentials suitable for our setting.

\begin{example}
Consider the functions
\[G(x,y)=x^2+(x-y)^4, \quad K(t,x,y)=(2+\sin{t})G(x,y)+|x^2+y^2|^2\cos^2{t}
\]
\[
W(t,x,y)=\frac{|x^2+y^2|^{5/2}(e^{t^2(x^2+y^2-1)}-1)}{t^2+1}+\sin{t}.
\]
$G$ is differentiable G-function satisfying $\Delta_2$ and $\nabla_2$ globally. Here $V=K-W$ satisfies \ref{V:C1K-W}-\ref{V:0}, where  $p_G=2$, $q_G^{\infty}=q_G=4$,
$\mu=5$, $a(t)=\sin{t}$, $b=2$, $\kappa(t)\geq 5\sin{t}$. On the other hand $K$ does not satisfy assumption $(H_1)$ and $W$ does not satisfy assumption  $(H_2)$ from \cite{Dao16}.
\end{example}

The next example shows that our results generalize Lemma 7 from \cite{Ter12}.

\begin{example}
Potential \[V(t,x)=a(t)G(x)-\lambda b(t)F(x),\] where $F$ is convex function satisfying \ref{delta2} globally, $G\prec\prec F$, the functions $a(t)$, $b(t)$ are continuously differentiable, strictly positive, even on $\R$,
$0< a\leq a(t)\leq A$, $0< b \leq b(t) \leq B$, $ta'(t)>0$  for $t\neq0$ and $tb'(t)<0$  for $t\neq0$ satisfies conditions \ref{V:C1K-W}-\ref{f:G*normf<Cf}.

Theorems \ref{thm:main_rho0<1} and \ref{thm:main_rho>2} assert the existence of periodic solutions of
\[
\begin{split}
&	\frac{d}{dt}\nabla G(\dot u)-a(t)\nabla G(u)+\lambda b(t)\nabla F(u)=f(t),\\
&u(-T)=u(T)=0
\end{split}
\]
which is a generalization of the problem (2) 	from \cite{Ter12}.

\end{example}
\section{Some facts about G-functions and Orlicz-Sobolev spaces}
\label{sec:Gfacts}

Assume that $G\colon \R^N\to [0,\infty)$ satisfies assumption \ref{V:G}. We say that
\begin{itemize}
\item
$G$ satisfies the $\Delta_2$ condition if
\begin{equation}\tag{$\Delta_2$}\label{delta2}
\exists_{K_1>2}\ \exists_{M_1\geq 0}\ \forall_{|x|\geq M_1}\ G(2x)\leq K_1G(x),
\end{equation}
\item
$G$ satisfies the $\nabla_2$ condition if
\begin{equation}\tag{$\nabla_2$}\label{nabla2}
\exists_{K_2>1}\ \exists_{M_2\geq 0}\ \forall_{|x|\geq M_2}\ G(x)\leq \frac{1}{2K_2}G(K_2x).
\end{equation}
\item $G$ satisfies $\Delta_2$ (resp. $\nabla_2$) globally if $M_1=0$ (resp. $M_2=0$).
\end{itemize}

Functions $G_1(x)=|x|^p$, $G_2(x)=|x|^{p_1}+|x|^{p_2}$ satisfy \ref{delta2} and \ref{nabla2}
globally. 
If $G$ does not satisfy \ref{delta2} globally, then it could decrease very fast near zero.
Function
\[
G(x)=\begin{cases}	|x|^2e^{-1/|x|}&x\neq 0\\0&x=0\end{cases}
\]
satisfies \ref{delta2} but does not satisfy \ref{delta2} globally.  For more details about $\Delta_2$
condition in case of N-function we refer the reader to \cite{KraRut61}.

Since $G$ is differentiable and convex,
\begin{equation}
\label{eq:G-G<inner}
G(x)-G(x-y)\leq \inner{\nabla G(x)}{y}\leq G(x+y)-G(x)\quad  \text{for all $x,y\in \R^N$}.
\end{equation}

A function $G^\ast(y)=\sup_{x\in\R^N}\{\inner{x}{y}-G(x)\}$ is called the Fenchel conjugate of
$G$. As an immediate consequence of definition we have the Fenchel inequality:
\begin{equation*}
\label{ineq:Fenchel}
\forall_{x,y\in\R^N}\ \inner{x}{y}\leq G(x)+G^{\ast}(y).
\end{equation*}

Now we briefly recall a notion of anisotropic Orlicz space. For more details we refer the reader to
\cite{Sch05} and \cite{ChmMak17}. The Orlicz space associated with $G$ is defined to be
\[
\LGspace = \{u\colon [-T,T]\to \R^N\colon \int_{-T}^T G(u) \,dt<\infty\}.
\]
The space $\LGspace$ equipped with the Luxemburg norm
\[
\LGnorm{u}=\inf\left\{\lambda>0\colon \int_I G\left(\frac{u}{\lambda}\right)\, dt\leq 1 \right\}.
\]
is a reflexive Banach space. We have the H\"older inequality
\[
\int_I \inner{u}{v}\,dt\leq 2\LGnorm{u}\LGnorm[G^\ast]{v}
\quad\text{ for every $u\in \LGspace$ and $v\in \LGspace[G^\ast]$.}
\]
Let us denote by
\[
\WTLGspace:=\left\{u\in \LGspace: \dot{u}\in \LGspace \text{ and } u(-T)=u(T)\right\}
\]
an anisotropic Orlicz-Sobolev space of periodic vector valued functions with norm
\[
\WLGnorm{u}=\LGnorm{u}+\LGnorm{\dot u}.
\]
We will also consider an equivalent norm given by
\[
\WLGnormother{u}=
\inf\left\{\lambda>0\colon
\int_{-T}^T G\left(\frac{u}{\lambda}\right)+
G\left(\frac{\dot u}{\lambda}\right)\, dt\leq 1
\right\}.
\]
The following proposition will be crucial to Lemma \ref{lem:J>alpha_rho<1}.
\begin{proposition}
\label{prop:normequivalence}
\[
\frac{1}{2}\WLGnorm{u}\leq\WLGnormother{u}\leq 2\WLGnorm{u}
\]
\end{proposition}
The proof for isotropic case can be found in \cite[Proposition 9, p.177]{RadRep15} it remains the
same for anisotropic case.

Functional $R_G(u):=\int_{-T}^{T}G(u)\,dt$ is called modular. The following result was proved in
\cite[Proposition 2.7]{Le14}.
\begin{proposition}
\label{prop:RGcoercive}
$R_G(u)$ is coercive on $\LGspace$ in the following sense:
\[
\lim_{\LGnorm{u}\to\infty}\frac{R_G(u)}{\LGnorm{u}}=\infty.
\]
\end{proposition}

Define the Simonenko indices for G-function
\[
p_G=\inf_{|x|>0}\frac{\inner{x}{\nabla G(x)}}{G(x)},
\quad
q_G=\sup_{|x|>0}\frac{\inner{x}{\nabla G(x)}}{G(x)},
\]
\[
q_G^{\infty}=\limsup_{|x|\to\infty}\frac{\inner{x}{\nabla G(x)}}{G(x)}.
\]
It is obvious that $p_G\leq q_G^{\infty}\leq q_G$.
Moreover, since $G$ satisfies $\Delta_2$ and $\nabla_2$, $1<p_G$ and $q_G<\infty$. 
The following results is crucial to Lemma \ref{lem:J>alpha_rho<1}.
\begin{proposition}\label{prop:norm_mod}
Let $G$ satisfies \ref{delta2} and \ref{nabla2} globally.
\begin{enumerate}
\item If $\LGnorm{u}\leq 1,$ then
$	\LGnorm{u}^{q_G}\leq R_G(u).$
\item 	If $\LGnorm{u}>1,$ then
$	\LGnorm{u}^{p_G}\leq R_G(u).$
\end{enumerate}
\end{proposition}
The proof can be found in appendix. 
More information about indices for isotropic case can be found in \cite{Sim64}, \cite{Mal89} and \cite{Cle04}.
For relations between Luxemburg norm and modular for anisotropic
spaces we refer the reader to \cite[Examples 3.8 and 3.9]{ChmMak17}.

For, respectively, continuous and compact embeddings we will use the symbols $\hookrightarrow$ and
$\hookrightarrow\hookrightarrow$. By $\imb{E}{F}$ we will denote the embedding constant for
$E\hookrightarrow F$. If $E=\Lpspace[s]$ we will write $s$ instead of $\Lpspace[s]$.

Let $G_1$ and $G_2$ be G-functions. Define
\begin{equation}
\label{eq:G1precG2}
G_1\prec G_2
\iff
\exists_{M\geq 0}\ \exists_{K>0}\ \forall_{|x|\geq M}\ G_1(x)\leq G_2(K\,x).
\end{equation}
The relation $\prec$ allows to compare growth rate of functions $G_1$ and $G_2$.

It is well known that if $G_1\prec G_2$, then $\LGspace[G_2]\hookrightarrow\LGspace[G_1]$. Let
$u\in\WLGspace$, $A_G:\R^N\to[0,\infty)$ be the greatest convex radial minorant of $G$ (see
\cite{AciMaz17}). Then
\[
\LGnorm[\infty]{u}\leq\imbinftG \WLGnorm{u},
\]
where $\imbinftG=A_G^{-1}\left(\frac{1}{2T}\right)\max\{1,2T\}$.

The following proposition  will be used in the proof of Lemma \ref{lem:ps}.
\begin{proposition}\label{prop:adelkowyTrik}(cf. \cite{Dao16})
For any $1<p\leq q<\infty$, such that $|\cdot|^p\prec G(\cdot)\prec |\cdot|^q$,
\[
\begin{split}
\int_{-T}^{T}|u|^{p}\,dt\geq \imb{\infty}{\WTLGspace}^{p-q
}\imb{G}{q}^{-q}\WLGnorm{u}^{p-q}\LGnorm{u}^{q}
\end{split}
\]
for  $u\in\WLGspace\backslash\{0\}$.
\end{proposition}
\begin{proof}
Let $u\in\WTLGspace\backslash\{0\}$. Since $G\prec |\cdot|^q$,
\[
\int_{-T}^{T}|u|^q\,dt=\LGnorm[q]{u}^q\geq \imb{G}{q}^{-q}\LGnorm{u}^q.
\]
From H\"older's inequality and embedding $\WLGspace\hookrightarrow\Lpspace[\infty]$ we obtain
\begin{multline*}
\int_{-T}^{T}|u|^q\,dt= \int_{-T}^{T}|u|^p|u|^{q-p},dt
\leq \\ \leq
\LGnorm[\infty]{u}^{q-p}\int_{-T}^{T}|u|^p\,dt\leq
(\imb{\infty}{\WTLGspace}\WLGnorm{u})^{q-p}\int_{-T}^{T}|u|^p\,dt.
\end{multline*}
\end{proof}


\section{Proof of the main results}
Let	$\JSymbol:\WTLGspace\to\R$ be given by
\begin{equation}
\label{J}
\tag{$\JSymbol$}
\JSymbol(u)=\int_{-T}^{T} G(\dot u)+K(t,u)-W(t,u)+\inner{f}{u}\,dt.
\end{equation}
From \ref{V:G}, \ref{V:C1K-W} and \cite[Thm. 5.5]{ChmMak17}) we have $\JSymbol\in C^1$ and
\begin{equation}
\label{J'}
\tag{$\JSymbol'$}
\JSymbol'(u)\varphi=\int_{-T}^T\inner{\nabla G(\dot{u})}{\dot{\varphi}}\,dt+\int_{-T}^T\inner{V_x(t,u)+f(t)}{\varphi}\,dt.
\end{equation}
It is standard to prove that critical points of $\JSymbol$ are solutions of \eqref{eq:ELT}.

Our proof is based on well-known Mountain Pass Theorem (see \cite{AmbRab73}]).

\begin{theorem}
\label{thm:mpt}
Let $X$ be a real Banach space and $I\in C^1(X,\R)$ satisfies the following conditions:
\begin{enumerate}
\item $I(0)=0$,
\item $I$ satisfies Palais-Smale condition,
\item there exist $\rho>0$, $e\in X$ such that $\norm{e}_X>\rho$ and $I(e)<0$,
\item there exists $\alpha>0$ such that $I\rvert_{\partial B_{\rho}(0)}\geq \alpha.$
\end{enumerate}
Then $I$ possesses a critical value  $c\geq\alpha$ given by
\[
c=\inf_{g\in\Gamma}\max_{s\in[0,1]}I(g(s)),
\]
where $\Gamma=\{g\in C([0,1],X);~~ g(0)=0,~ g(1)=e\}.$
\end{theorem}

It remains to prove that \ref{J} satisfies all the assumptions of the Mountain Pass Theorem. We
divide the proof into sequence of lemmas.

\begin{lemma}
\label{lem:ps}
$\JSymbol$ satisfies the Palais-Smale condition, i.e. every sequence $\{u_n\}\subset \WTLGspace $
such that  $\{\JSymbol(u_n)\}$ is bounded and $\JSymbol'(u_n)\to 0$ as $n\to\infty$ contains a
convergent subsequence.
\end{lemma}
\begin{proof}
From \ref{V:AR} and  \eqref{J'} we get
\begin{multline}
\label{eq:ARJ'}
\int_{-T}^{T}\mu W(t,u)-(q_G^{\infty}+\nu)K(t,u)\,dt
\leq\\\leq
-\JSymbol'(u)u + \int_{-T}^{T}\inner{\nabla G(\dot u)}{\dot u}\,dt +
\int_{-T}^{T}\inner{f(t)}{u}+\kappa(t)\,dt.
\end{multline}
From the definition of the functional we obtain
\begin{multline*}
\mu\int_{-T}^{T}G(\dot{u})\,dt + (\mu-q_G^{\infty}-\nu)\int_{-T}^{T}K(t,u)\,dt
=\\=
\mu\JSymbol(u)+\int_{-T}^{T} \mu W(t,u)-(q_G^{\infty}+\nu)K(t,u)\,dt -
\mu\int_{-T}^{T}\inner{f(t)}{u}\,dt.
\end{multline*}
Applying \eqref{eq:ARJ'}, the H\"older inequality and \ref{f:G*normf<Cf} we have
\begin{multline}
\label{eq:coupl}
\mu\int_{-T}^{T}G(\dot u)\,dt+(\mu-q_G^{\infty}-\nu)\int_{-T}^{T}K(t,u)\,dt
\leq\\ \leq
\mu\JSymbol(u)-\JSymbol'(u)u +
\int_{-T}^{T}\inner{\nabla G(\dot u)}{\dot u}\,dt+C_{\kappa}+(1-\mu)C_f\WLGnorm{u},
\end{multline}
where $C_{\kappa}=\int_{-T}^{T}\kappa(t)\,dt$ and $C_f=2\LGnorm[G^{\ast}]{f}$.
From the definition of $q_G^{\infty}$ there exists $M>0$ such that
\begin{equation}
\label{eq:nablaG<qGinf}
\inner{x}{\nabla G(x)}\leq (q_G^{\infty}+\nu)\,G(x), \text{ for $|x|>M$}.
\end{equation}

Hence
\begin{multline}
\label{gradG:estinf}
\int_{-T}^{T}\inner{\nabla G(\dot u)}{\dot u}\,dt
=
\int_{\{|\dot u|\geq M\}}\inner{\nabla G(\dot u)}{\dot u}\,dt
+
\int_{\{|\dot u|< M\}}\inner{\nabla G(\dot u)}{\dot u}\,dt
\leq\\ \leq
(q_G^{\infty}+\nu)\int_{\{|\dot u|\geq M\}}G(\dot u)\,dt + C_{\nabla G}
\leq
(q_G^{\infty}+\nu)\int_{-T}^{T}G(\dot u)\,dt+C_{\nabla G},
\end{multline}
where $C_{\nabla_G}=\sup_{|x|<M}2\,T\,M\,\nabla G(x)$. Applying \eqref{gradG:estinf}
we can rewrite \eqref{eq:coupl} as
\begin{multline}
\label{eq:przedostatnie}
(\mu-q_G^{\infty}-\nu)\int_{-T}^{T}G(\dot u)+K(t,u)\,dt
\leq\\ \leq
\mu\JSymbol(u)-\JSymbol'(u)u+C_{\kappa}+(1-\mu)C_f\WLGnorm{u}+C_{\nabla G}.
\end{multline}
From \ref{V:infty}, given any $0<\varepsilon_1<b_1$, there exists $\delta_1\geq0$  such that
\[
K(t,x)\geq (b_1-\varepsilon_1)|x|^p-\delta_1, \quad\text{ for $x\in\R^N$.}
\]
By Proposition \ref{prop:adelkowyTrik} we obtain
\begin{multline}
\label{eq:K}
\int_{-T}^{T}K(t,u)\,dt\geq\int_{-T}^{T}(b_1-\varepsilon_1)|u|^p \,dt-2T\delta_1
\geq \\ \geq
(b_1-\varepsilon_1)\imb{\infty}{\WLGspace}^{p-q}
\imb{G}{q}^{-q}\frac{\LGnorm{u}^q}{\WLGnorm{u}^{q-p}}-2T\delta_1,
\end{multline}
for any $q$ such that $G\prec|\cdot|^q$. Finally, applying \eqref{eq:K} to \eqref{eq:przedostatnie}
we obtain
\begin{multline}
\label{eq:ps_finalestimationforu}
(\mu-q_G^{\infty}-\nu)
\left(
\int_{-T}^{T}G(\dot u)\,dt
+
(b_1-\varepsilon_1)\imb{\infty}{\WLGspace}^{p-q}
\imb{G}{q}^{-q}\frac{\LGnorm{u}^q}{\WLGnorm{u}^{q-p }}
\right)
+\\-
(\mu-1)C_f\WLGnorm{u}+\JSymbol'(u)u\leq\mu\JSymbol(u)+C_{\kappa}+C_{\nabla_G}+C_{\delta_1},
\end{multline}
where $C_{\delta_1}=(2T\delta_1)(\mu-q_G^{\infty}-\nu)$.

Let $\{u_n\}\subset\WTLGspace$ will be a Palais-Smale sequence for $\JSymbol$. 	There exist $C_J$,
$C_{J'}>0$ such that
\begin{equation}
\label{eq:CJJ'}
\JSymbol(u_n)\leq C_J,\quad \JSymbol'(u_n)u_n\geq- C_{J'}\WLGnorm{u_n}
\end{equation}
Without loss of generality we can assume, that $\WLGnorm{u_n}>0$. Substituting $u_n$ into
\eqref{eq:ps_finalestimationforu} and by \eqref{eq:CJJ'} we obtain
\begin{equation}
\label{eq:un<C}
\WLGnorm{u_n}\left(\frac{R_G(\dot
u_n)}{\WLGnorm{u_n}}+\frac{\LGnorm{u_n}^{q}}{\WLGnorm{u_n}^{1+q-p}}-C'\right)\leq C'',
\end{equation}
where $C',C''$ are suitable constants independent of $n$.

We show that $\{u_n\}$ is bounded. On the contrary, suppose that there exists a subsequence of $u_n$
(still denoted $u_n$) such that $\WLGnorm{u_n}\to\infty$. Consider three cases.

\begin{enumerate}
\item  Let $\LGnorm{u_n}\to\infty$ and $\LGnorm{\dot u_n}\to\infty$ (again, w.l.o.g.  $\LGnorm{\dot
u_n}>0$). From Proposition \ref{prop:RGcoercive} we have that
\begin{multline*}
\frac{R_G(\dot u_n)}{\WLGnorm{u_n}}+\frac{\LGnorm{u_n}^{q}}{\WLGnorm{u_n}^{1+q-p}}
= \\ =
\frac{R_G(\dot u_n)}{\LGnorm{\dot u_n}}\frac{\LGnorm{\dot u_n}}{\WLGnorm{u_n}}
+
\left(\frac{\LGnorm{u_n}}{\WLGnorm{u_n}}\right)^{1+q-p}\LGnorm{u_n}^{p-1}\to\infty.
\end{multline*}

\item Let $\LGnorm{\dot u_n}\to\infty$ and $\LGnorm{u_n}$ is bounded. Then
\[
\frac{R_G(\dot u_n)}{\WLGnorm{u_n}}=\frac{R_G(\dot u_n)}{\LGnorm{u_n}+\LGnorm{\dot
u_n}}=\frac{\frac{R_G(\dot u_n)}{\LGnorm{\dot u_n}}}{\frac{\LGnorm{u_n}}{\LGnorm{\dot u_n}}+1} \to
\infty \text{ as } \LGnorm{\dot u_n}\to\infty.
\]

\item Let $\LGnorm{ u_n}\to\infty$ and  $\LGnorm{ \dot u_n}$ is bounded. Since $p>1$, we have
\[
\frac{R_G(\dot u_n)}{\WLGnorm{u_n}}+\frac{\LGnorm{u_n}^{q}}{\WLGnorm{u_n}^{1+q-p}}
\geq
\frac{\LGnorm{u_n}^q}{ (\LGnorm{u_n}+\LGnorm{\dot u_n})^{1+q-p}}\to\infty.
\]
\end{enumerate}
Therefore, in view of \eqref{eq:un<C}, $\{u_n\}$ is bounded in $\WTLGspace$.

It follows from reflexivity of $\WTLGspace$ and embeddings
$\WTLGspace \hookrightarrow\hookrightarrow\Lpspace[G]$,
$\WTLGspace \hookrightarrow \hookrightarrow\Wspace{1,1}$
that there exists $u\in\WTLGspace$ and a subsequence of $u_n$ (still denoted $u_n$) such that
$u_n\to u$ in $\LGspace$, $\dot u_n\to \dot u$ in $\Lpspace[1]$ and hence pointwise a.e.


Since $\{u_n\}$ is a Palais-Smale sequence, we have
\[
0\leftarrow \JSymbol' (u_n) (u_n-u)
=
\int_{-T}^T \inner{\nabla G(\dot u_n) }{ \dot u_n-\dot u}\,dt
+
\int_{-T}^T\inner{V_x(t,u_n)+f(t)}{u_n-u}\,dt.
\]
Since  $\int_{-T}^T\inner{V_x(t,u_n)+f(t)}{u_n-u}\,dt\to 0$ we can deduce that
\[
\int_{-T}^T \inner{\nabla G(\dot u_n) }{ \dot u_n-\dot u} \, dt \to 0.
\]
From \eqref{eq:G-G<inner} we obtain
\[
\int_{-T}^T G(\dot u_n) \, dt \leq \int_{-T}^T G(\dot u)\, dt +
\int_{-T}^T\inner {\nabla G(\dot u_n)}{ (\dot u_n-\dot u)}\, dt
\]
Hence
\[
\limsup_{n\to +\infty} \int_{-T}^T G(\dot u_n) \, dt \leq \int_{-T}^T G(\dot u)\, dt.
\]
On the other hand, by Fatou's Theorem we have
\[
\liminf_{n\rightarrow +\infty} \int_{-T}^{T} G(\dot u_n)\, dt \geq \int_{-T}^{T} G(\dot u)\, dt.
\]
Combining these inequalities we get that
\[
\int_{-T}^{T} G(\dot u_n)\, dt \to \int_{-T}^{T} G(\dot u)\,dt.
\]
Since norm convergence is equivalent to modular convergence, $\dot u_n\to \dot u$ in $\LGspace$.
\end{proof}

We next prove that $\JSymbol$ is negative for some point outside $B_\rho(0)$.

\begin{lemma}
\label{lem:Je<0}
There exist $e\in\WTLGspace$ such that $\WLGnorm{e}>\rho$ and $\JSymbol(e)<0$.
\end{lemma}
\begin{proof}
By assumption \ref{V:infty}, there exist $\varepsilon_0,r>0$ such that
\[
W(t,x)\geq(3+\varepsilon_0)\max\{K(t,x),G(x)\}\quad \text{for $|x|>r$}.
\]
which gives
\begin{equation}
\label{ineq:W-K}
K(t,x)-W(t,x)\leq(2+\varepsilon_0)G(x)\quad \text{for $|x| > r$}.
\end{equation}
Fix $v\in\R^N$. For $\xi>T+1$ define $e\colon [-T,T]\to\R^N$ by
\[
e(t)=\xi\left(1-\frac{|t|}{T+1}\right)v.
\]
Direct computation shows
\[
\dot e(t) =
\begin{cases}
-\frac{\xi}{T+1}\,v, & \text{$t\in (0,T],$}\\
\frac{\xi}{T+1}\,v, & \text{$t\in [-T,0).$}
\end{cases}
\]
Since $\LGnorm[\infty]{e}=\xi>T+1$ and $\LGnorm[\infty]{\dot e}=\xi/(T+1)>1$, we can choose $\xi$
such that both \eqref{ineq:W-K} and $\WLGnorm{e}\geq \rho$ hold. From \eqref{J}, the Fenchel
inequality and \eqref{ineq:W-K} we have
\begin{multline*}
\JSymbol(e) \leq \int_{-T}^{T}G(\dot e)+K(t,e)-W(t,e)+G(e)+G^{\ast}(f)\,dt
\leq \\ \leq
\int_{-T}^{T}G(\dot e)-G(e)-\varepsilon_0G(e)+G^{\ast}(f)\,dt.
\end{multline*}
Since $1-\tfrac{|t|}{T+1}\geq \tfrac{1}{T+1}$ for $t\in[-T,T]$, we have
\[
\int_{-T}^{T}G(\dot e)-G(e)\,dt
=
\int_{-T}^T G\left(\frac{\xi}{T+1} v\right) - G\left(\xi\left(1-\frac{|t|}{T+1}\right) v
\right)\,dt \leq 0.
\]
Choosing $\xi$ large enough we get
\[
\JSymbol(e)\leq \int_{-T}^{T}-\varepsilon_0\,G(e)+G^{\ast}(f)\,dt < 0.
\]
\end{proof}

In order to show that $\JSymbol$ satisfies the fourth assumption of Mountain Pass Theorem, we first
provide some estimates for $\modularG{\dot{u}}+\modularG{u}$ on $\partial B_\rho(0)$.

If \ref{delta2} and \ref{nabla2} are satisfied globally then we can use Proposition
\ref{prop:norm_mod} to estimate $\modularG{\dot{u}}+\modularG{u}$ from below by
$(\rho/2)^r$, $r>1$, for any $\rho> 0$. If $G$ does not satisfies \ref{delta2} and \ref{nabla2}
globally then we cannot use Proposition \ref{prop:norm_mod} (for explanation see Remark
\ref{rem:apinfinityestGlambda}). In this case we use equivalent norm and Proposition
\ref{prop:normequivalence} but we obtain only that $\modularG{\dot{u}}+\modularG{u}\geq \rho/2$.
Moreover, we are forced to assume $\rho > 2$.

Let $u\in\WTLGspace$ be such that $\WLGnorm{u}=\rho$. Set $\rho_1=\LGnorm{u}$, $\rho_2=\LGnorm{\dot
u}$, $\rho_1+\rho_2=\rho$. Assuming that $G$ satisfies \ref{delta2} an \ref{nabla2} globally
we get by Proposition \ref{prop:norm_mod} the following estimates:
\begin{enumerate}
\item  If $\rho_1,\rho_2\leq 1$ then
$
R_G(\dot u)+R_G(u)\geq \LGnorm{\dot u}^{q_G}+\LGnorm{u}^{q_G}.
$
Hence
\begin{equation}
\label{eq:rho12<1}
R_G(\dot u)+R_G(u)\geq 2^{1-q_G}(\LGnorm{\dot u}+\LGnorm{ u})^{q_G}
\geq \left(\rho/2\right)^{q_G}.
\end{equation}
since ${\rho_1}^{q_G}+\rho_2^{q_G}\geq 2^{1-{q_G}}(\rho_1+\rho_2)^{q_G}$.

\item If $\rho_1\leq 1$, $\rho_2\geq 1$ then
$
\left(\rho_1+\rho_2\right)^{p_G}\leq (2\rho_2)^{p_G}\leq
2^{p_G}\left(\rho_1^{q_G}+\rho_2^{p_G}\right).
$
Hence
\begin{equation}
\label{eq:rho1<1rho2>1}
R_G(\dot u)+R_G(u)\geq \LGnorm{\dot u}^{p_G}+\LGnorm{u}^{q_G}\geq (\rho/2)^{p_G}.
\end{equation}

\item If $\rho_1\geq 1$, $\rho_2\leq 1$ then
$
(\rho_1+\rho_2)^{p_G}\leq (2\rho_1)^{p_G}\leq 2^{p_G}\left(\rho_1^{p_G}+\rho_2^{q_G}\right).
$
Thus
\begin{equation}
\label{eq:rho1>1rho2<1}
R_G(\dot u)+R_G(u)\geq \LGnorm{\dot u}^{q_G}+\LGnorm{u}^{p_G}\geq (\rho/2)^{p_G}.
\end{equation}

\item If $\rho_1, \rho_2\geq 1$ then
\begin{equation}
\label{eq:rho12>1}
R_G(\dot u)+R_G(u)\geq  \LGnorm{\dot u}^{p_G}+\LGnorm{u}^{p_G}\geq (\rho/2)^{p_G}.
\end{equation}
\end{enumerate}

From the other hand, proposition \ref{prop:normequivalence} furnishes the bound
\[
\inf\left\{\lambda>0\colon \int_{-T}^T G\left(\frac{u}{\lambda}\right)+ G\left(\frac{\dot
u}{\lambda}\right)\, dt\leq 1 \right\}
\geq \frac{1}{2}\rho.
\]
Therefore
\[
\int_{-T}^T G\left(\frac{2u}{\rho}\right)+ G\left(\frac{2\dot u}{\rho}\right)\, dt\geq 1.
\]
and consequetly,
\begin{equation}
\label{eq:mod>rho}
R_G(u)+ R_G(\dot u)\geq \frac{\rho}{2}
\end{equation}
provided $\rho>2$.

\begin{lemma}
\label{lem:J>alpha_rho<1}
Assume that either \eqref{thm1:as1} or \eqref{frho>2} holds. There exists positive constant
$\alpha$ such that $\JSymbol\rvert_{\partial B_{\rho}(0)}\geq
\alpha.$
\end{lemma}
\begin{proof}
From the definition of $\rho$ and embedding $\WTLGspace \hookrightarrow\Lpspace[\infty]$  we have
\[
|u(t)|\leq\LGnorm[\infty]{u}\leq\imb{\infty}{\WTLGspace}\WLGnorm{u}
=
\imb{\infty}{\WTLGspace}\rho=\rho_0\quad \text{for $t\in[-T,T].$}
\]
From \ref{V:near0} and the Fenchel inequality we obtain
\begin{multline*}
\JSymbol(u)
\geq
\int_{-T}^{T}G(\dot{u})+bG(u)-a(t)+\inner{f}{u}\,dt
\geq\\\geq
\min\{1,b-1\}(R_G(\dot{u})+R_G(u))-R_{ G^{\ast}}(f)-\int_{-T}^{T}a(t)\,dt.
\end{multline*}

Assume that \eqref{thm1:as1} holds. If $\rho\leq 2$ then $(\rho/2)^{p_G}\geq (\rho/2)^{q_G}$,
\eqref{eq:rho12<1}, \eqref{eq:rho1<1rho2>1} and \eqref{eq:rho1>1rho2<1} yields
\[
\JSymbol(u)
\geq
\min\{1,b-1\}(\rho/2)^{q_G} -R_{G^{\ast}}(f)-\int_{-T}^{T}a(t)\,dt =: \alpha.
\]
If  $\rho> 2$,  then by \eqref{eq:rho1<1rho2>1}, \eqref{eq:rho1>1rho2<1} and \eqref{eq:rho12>1}
we get
\[
\JSymbol(u)\geq \min\{1,b-1\}(\rho/2)^{p_G} -R_{G^{\ast}}(f)-\int_{-T}^{T}a(t)\,dt>0 =:\alpha.
\]
From \eqref{thm1:as1} it follows that in both cases $\alpha>0$.

Assume that \eqref{frho>2} holds. From \eqref{eq:mod>rho} we obtain
\begin{multline*}
\JSymbol(u)
\geq
\min\{1,b-1\}(R_G(\dot u)+R_G(u))-R_{G^{\ast}}(f)-\int_{-T}^{T} a(t)
\geq\\\geq
\min\{1,b-1\}(\rho/2)-R_{G^{\ast}}(f)-\int_{-T}^{T} a(t)\,dt=:\alpha.
\end{multline*}
From \eqref{frho>2} we have $\alpha>0$.
\end{proof}

Now we are in position to prove our main theorems. Note that by \ref{V:0} and $G(0)=0$ we have
$\JSymbol(0)=0$. From Lemmas \ref{lem:ps}, \ref{lem:Je<0} and \ref{lem:J>alpha_rho<1}   we have
that $\JSymbol$ satisfies all assumptions of the Mountain Pass Theorem. Hence there exists a
critical point $u\in\WTLGspace$ of $\JSymbol$ and \eqref{eq:ELT} have periodic solution.

Actually, we can show that any solution to \eqref{eq:ELT} is more regular (cf. Corollary 16.16 in
\cite{Cla13}).

\begin{proposition}
If $u\in\WTLGspace$ is a solution of \eqref{eq:ELT}, then $u\in\WLGspace[\infty]$.
\end{proposition}
\begin{proof}
Let $u\in\WLGspace$ be a solution of \eqref{eq:ELT}. Then
\[
\nabla G(\dot{u}(t))=\int_{-T}^{t}\nabla V(t,u(t))\,dt+C
\]
and there exists $M>0$ such that $|\nabla G(\dot u(t))|\leq M<\infty.$ From the other hand
\[
G(\dot u(t))\leq\inner{\nabla G(\dot u(t))}{\dot u(t)}\leq M|\dot u(t)|.
\]
Since $\frac{G(v)}{|v|}\to\infty $ as $|v|\to\infty$, we obtain $|\dot u(t)|$ is bounded.
\end{proof}

\begin{remark}
If $G$ is strictly convex then one can show that if $u\in\WTLGspace$ is a solution of
\eqref{eq:ELT}, then $u\in\mathbf{C}^1$.
\end{remark}

\begin{remark}
Theorem \ref{thm:main_rho0<1} remains true if we change  assumption \eqref{thm1:as1}  to
\begin{equation}
\label{eq:RGf_pat1}
R_G^{\ast}(f)+\int_{-T}^{T}a(t)\,dt<\min\{1,b-1\}\begin{cases}  2(\rho/2)^{q_G},&\rho\leq
2^{1-1/(q_G-p_G)}\\(\rho/2)^{p_G}&\rho> 2^{1-1/(q_G-p_G)}.\end{cases}
\end{equation}
Estimate in the first case is better than \eqref{thm1:as1} but it is taken  on smaller set. In the
second case estimate is the same as in \eqref{thm1:as1} but can be taken on bigger set.
\end{remark}

\begin{remark}
In the proof of Lemma \ref{lem:J>alpha_rho<1} we can use the H\"older inequality instead of the
Fenchel inequality to estimate $\int_{-T}^{T}\inner{f}{u}\,dt$. It allows us to take $b>0$ if
$\rho\leq1.$
\end{remark}


\appendix

\section{}

\renewcommand{\appendixname}{}

Assume that $G$ satisfies \ref{nabla2} globally. It is easy to show that $G^{\ast}$ satisfies
\ref{delta2} globally with $K_1^{\ast}=2K_2$.

Since $G\in\mathbf{C}^1$, we have
\[
K_1G(x)\geq G(2x)\geq G(2x)-G(x)\geq \inner{x}{\nabla G(x)}\quad \text{for all $x\in \R^N$.}
\]
Let $y\in \R^N$ and $s\in\partial G^{\ast}(y)$. Since  $G^{\ast}$ satisfies \ref{delta2} globally,
we have
\[
K_1^{\ast}G^{\ast}(y)\geq G^{\ast}(2y)\geq G^{\ast}(2y)-G^{\ast}(y)\geq\inner{s}{y}
\quad \text{for all $y\in \R^N$}
\]
Let $x\in \R^N$. Then $x\in\partial G^{\ast}(\nabla G(x))$ and $G(x)+G^{\ast}(\nabla G(x))=
\inner{x}{\nabla G(x)}$. It follows that
\[
G(x)=\inner{x}{\nabla G(x)}-G^{\ast}(\nabla G(x))
\leq
\left(1-\frac{1}{K_1^{\ast}}\right)\inner{x}{\nabla G(x)}.
\]
Finally
\begin{equation}
\label{ineq:alem:nablaG/G}
\frac{2K_2}{2K_2-1} \leq \frac{\inner{x}{\nabla G(x)}}{G(x)}\leq K_1
\quad \text{for all $x\in \R^N$.}
\end{equation}
Since $K_2>1$, we have $\frac{2K_2}{2K_2-1}>1$ and from \eqref{ineq:alem:nablaG/G} we obtain
\begin{equation*}
p_G>1\quad  \text{and}\quad q_G<\infty.
\end{equation*}
For any $x\in \R^N$ and $\lambda\geq 1$ we have
\[
\log{G(\lambda x)}-\log{G(x)}
=
\int_{1}^{\lambda}\frac{\inner{\nabla G(\lambda x)}{x}}{G(\lambda x)}\,d\lambda
\leq
\int_{1}^{\lambda}\frac{q_G}{\lambda}\,d\lambda=\log{\lambda^{q_G}}.
\]
Thus
\begin{equation}
\label{eq:Glambda<lambdaq_G}
G(\lambda x)\leq \lambda^{q_G}G(x)\quad \text{for all $x\in\R^N$, $\lambda\geq 1$}.
\end{equation}
Similarly, we show
\begin{equation}
\label{eq:Glambda<lambdap_G}
G(\lambda x)\geq \lambda^{p_G}G(x)\quad\text{ for all $x\in\R^N$, $\lambda\geq 1$.}
\end{equation}

\begin{lemma}
\label{lem:pG<R<qG}
Let $u\in \LGspace$.
\begin{enumerate}
\item If $\LGnorm{u}<1$ then $R_G(u)\geq\LGnorm{u}^{q_G}$.
\item If $\LGnorm{u}>1$ then $R_G(u)\geq\LGnorm{u}^{p_G}$.
\end{enumerate}
\end{lemma}
\begin{proof}
For any $0<\beta<\LGnorm{u}<1$ we have $R_G\left(\frac{u}{\beta}\right)\geq 1$. From
\eqref{eq:Glambda<lambdaq_G} we obtain that
$
G\left(\frac{x}{\beta}\right)\leq\left(\frac{1}{\beta}\right)^{q_G}G(x)$ for all $x\in\R^N$. Hence
\[
R_G(u)\geq\beta^{q_G}R_G\left(\frac{u}{\beta}\right)\geq \beta^{q_G}.
\]
Letting $\beta\uparrow\LGnorm{u}$ gives $R_G(u)\geq\LGnorm{u}^{q_G}.$

For any $1<\beta<\LGnorm{u}$ we have $R_G\left(\frac{u}{\beta}\right)> 1.$ Then from
\eqref{eq:Glambda<lambdap_G} we obtain that
$
G(x)\geq\beta^{p_G}G\left(\frac{x}{\beta}\right)$ for all $x\in\R^N$. Hence
\[
R_G(u)\geq\beta^{p_G}R_G\left(\frac{u}{\beta}\right)\geq \beta^{p_G}.
\]
Letting $\beta\uparrow\LGnorm{u}$ gives $R_G(u)\geq\LGnorm{u}^{p_G}.$
\end{proof}

\begin{remark}
\label{rem:apinfinityestGlambda}
If $G$ satisfies \ref{delta2} and \ref{nabla2} (not globally), estimations similar to
\eqref{eq:Glambda<lambdaq_G} and \eqref{eq:Glambda<lambdap_G} can be obtained for sufficiently
large $|x|$. However, even if $\LGnorm{u}$ is large it does not necessarily mean that $|u(t)|$.
Hence we cannot use these estimates to obtain result similar to Lemma \ref{lem:pG<R<qG}.
\end{remark}

\bibliographystyle{elsarticle-num}
\bibliography{PerSolAniOSspace}


\end{document}